\documentclass{article}
\usepackage{amssymb,pstricks}
\usepackage{latexsym}
\usepackage{amsmath}
\usepackage{amsthm}
\topskip  \theoremstyle{remark}
\theoremstyle{plain}
\newtheorem{theorem}{Theorem}[section]
\newtheorem{lemma}[theorem]{Lemma}

\newtheorem{proposition}[theorem]{Proposition}

\newtheorem{corollary}[theorem]{Corollary}

\begin{document}
%-----------------------------------------------------------
\title{Avoidance of vincular patterns by flattened derangements}
\author{Toufik Mansour\\
\small Department of Mathematics, University of Haifa, 3498838 Haifa, Israel\\[-0.8ex]
\small\texttt{tmansour@univ.haifa.ac.il}\\[1.8ex]
Mark Shattuck\\
\small Department of Mathematics, University of Tennessee, 37996 Knoxville, TN\\[-0.8ex]
\small\texttt{mark.shattuck2@gmail.com}\\[1.8ex]
}
\date{\small }%\today}
\maketitle
%-----------------------------------------------------------

\begin{abstract}
In this paper, we consider the problem of avoiding a single vincular pattern of length three by derangements in the flattened sense and find explicit formulas for the generating functions enumerating members of each corresponding avoidance class according to the number of cycles.  We make frequent use of the kernel method in solving the functional equations that arise which are satisfied by these (ordinary) generating functions.  In the case of avoiding 23-1, which is equivalent to 32-1 in the flattened sense, it is more convenient to consider the exponential generating function instead due to the form of the recurrence.  This leads to an explicit expression for the distribution of the number of cycles in terms of Stirling numbers of the second kind and the determinant of a certain tridiagonal matrix.  Finally, the cases of 3-12 and 3-21 are perhaps the most difficult of all, and here we make use of a pair of auxiliary statistics in order to find a system of recurrences that enumerate each avoidance class.
\end{abstract}



\section{Introduction}

Let $\rho=\rho_1\cdots \rho_n$ and $\tau=\tau_1\cdots\tau_m$ be permutations of $[n]=\{1,\ldots,n\}$ and $[m]$, where $1 \leq m \leq n$.  Then $\rho$ is said to \emph{contain} an occurrence of $\tau$ if there exists a subsequence of $\rho$ that is order-isomorphic to $\tau$ and is said to \emph{avoid} $\tau$ otherwise.  In this context, $\tau$ is often referred to as a \emph{pattern}.  Suppose dashes are inserted between some of the letters of $\tau$ such that $\tau=\tau^{(1)}\text{-}\tau^{(2)}\text{-}\cdots\text{-}\tau^{(k)}$ for some $k\geq 1$, where $\tau^{(i)}$ for each $i \in [k]$ represents a sequence of consecutive letters of $\tau$ none of which are separated by dashes.  Then $\rho$ is said to contain $\tau$, as a \emph{vincular}, or \emph{dashed}, pattern if $\rho$ contains a subsequence $s$ that is isomorphic to $\tau$ such that any pair of letters of $s$ corresponding to adjacent entries within some $\tau^{(i)}$ must be adjacent in $\rho$.  If no such subsequence $s$ exists that is isomorphic to $\tau$ and meeting the adjacency requirements, then $\rho$ avoids $\tau$ as a vincular pattern.  If $a_i=|\tau^{(i)}|$ for $1 \leq i \leq k$, then the vincular pattern $\tau=\tau^{(1)}\text{-}\cdots\text{-}\tau^{(k)}$ is said to be of \emph{type} $(a_1,\ldots,a_k)$. For example, the permutation $\rho=4315762 \in \mathcal{S}_7$ contains an occurrence of the vincular pattern 23-1, as witnessed by the subsequence 572.  However, $\rho$ avoids the pattern 3-12, though it contains a couple of subsequences that are isomorphic to 312.  A vincular pattern for which $a_i=1$ for all $i$ is referred to as \emph{classical}, meaning that there is no adjacency requirement for letters within an occurrence of the pattern.  Hence, in the preceding example, $\rho$ contains two occurrences of the classical pattern 3-1-2.  We refer the reader to the text by Kitaev \cite[Chapter~7]{Kit} for a full discussion of vincular and other types of patterns.

Suppose $\pi \in \mathcal{S}_n$ is expressed in standard cycle form, meaning that its cycles are written from left to right in ascending order of their first elements, with the first element the smallest in each cycle.  The \emph{flattening} of $\pi$, denoted by $\text{flat}(\pi)$, is then obtained by erasing all of the parentheses enclosing the cycles of $\pi$ and considering the resulting permutation in one-line notation.  For example, if $\pi \in \mathcal{S}_8$ has disjoint cycles $(784)$, $(62)$ and $(135)$, then its standard cycle form is $(135)(26)(478)$ and $\text{flat}(\pi)=13526478$.  Flattened permutations were apparently first considered by Carlitz \cite{Car} in his definition of a certain kind of inversion statistic on $\mathcal{S}_n$, which was further studied in \cite{Shat}.  Later, Callan \cite{C} introduced the concept of a flattened partition of $[n]$, where the idea is basically the same except now brackets enclosing the various blocks are erased so as to obtain the flattened form.  He considered the avoidance problem on flattened partitions and found the cardinality of each avoidance class corresponding to a single classical pattern of length three.  Further variants of the notion of flattening have been put forth on other discrete structures, such as Catalan words \cite{BHR}, Stirling permutations \cite{BEF} and parking functions \cite{EHM}.

In analogy with the definition given in \cite{C} for set partitions, we will say that $\pi \in \mathcal{S}_n$ contains or avoids a pattern $\tau$ \emph{in the flattened sense} if $\pi'=\text{flat}(\pi)$ contains or avoids $\tau$ in the usual sense. In this setting, $\tau$ may be any kind of pattern, i.e., classical, dashed or consecutive. For example, let $\pi=419738265 \in \mathcal{S}_9$, which has standard cycle form $(1472)(395)(68)$.  Then $\pi$ contains the vincular pattern 2-31 in the flattened sense as $\pi'=147239568$ contains an occurrence of the pattern (witnessed by the subsequence 795).  Further, it is seen that $\pi$ avoids 12-3 and contains 3-21 in the usual sense, though it avoids 3-21 and contains 12-3 in the flattened sense, with both $\pi$ and $\pi'$ containing 1-23.  In general, the avoidance or containment of a pattern by a permutation in one sense is not related to its avoidance or containment in the other.

A \emph{derangement} is a permutation containing no fixed points, i.e., each of its cycles is of length at least two.  We will denote the set of all derangements of $[n]$ by $\mathcal{D}(n)$.  Here, we consider the problem of avoidance of a vincular pattern of length three by members of $\mathcal{D}(n)$ in the flattened sense.  This extends prior work concerning pattern avoidance in the flattened sense as well as related statistics on such structures as permutations \cite{MSh,MSW}, set partitions \cite{C}, Catalan words \cite{BHR}, Stirling permutations \cite{BEF} and parking functions \cite{EHM}.  Our results also extend prior work dealing with the avoidance of classical patterns by derangements in both the usual \cite{MR,RSZ} and the flattened \cite{MSder} senses. Finally, the avoidance, along with the distribution, of vincular patterns of type (2,1) in the flattened sense on the set of all permutations of length $n$ was studied in \cite{MSW2}.

Given a pattern $\tau$, let $\mathcal{D}_\tau(n)$ denote the subset of $\mathcal{D}(n)$ whose members avoid $\tau$ in the flattened sense and let $d_\tau(n)=|\mathcal{D}_\tau(n)|$.  Let $d_\tau(n;y)$ be the distribution on $\mathcal{D}_\tau(n)$ for the statistic tracking the number of cycles (marked by the indeterminate $y$); that is,
$$d_\tau(n;y)=\sum_{\pi\in\mathcal{D}_\tau(n)}y^{\mu(\pi)}, \qquad n \geq 2,$$
with $d_\tau(1;y)=0$, where $\mu(\pi)$ denotes the number of (disjoint) cycles of $\pi$.  Note $d_\tau(n;1)=d_\tau(n)$, by the definitions.  We seek an explicit formula for $d_\tau(n;y)$ and/or its generating function in cases when $\tau$ is a vincular pattern of length three.

The organization of this paper is as follows.  In the next section, we consider the avoidance of a pattern of type (2,1) by derangements in the flattened sense.  The case 13-2 may be dealt with quickly. For three other patterns, we refine the avoidance class in question by considering the second or third letter in the flattened form, which enables one to write recurrences enumerating members of each class according to the number of cycles.  These recurrences lead to functional equations satisfied by the corresponding ordinary generating functions, for which we make use of the \emph{kernel} method (see, e.g., \cite{HM}) to ascertain an explicit solution.  See subsections \ref{s31-2} and \ref{s21-3}, where this strategy is applied to the patterns 31-2, 21-3 and 12-3.  On the other hand, to enumerate members of $\mathcal{D}_{23\text{-}1}(n)$, we find it more convenient to consider the exponential instead of the ordinary generating function due to the form of the recurrence in Lemma \ref{23-1lem1}.  This leads to a second-order linear differential equation with variable coefficients, which, though it does not appear to possess a simple closed-form solution, can be used to deduce an explicit formula for $d_{23\text{-}1}(n;y)$ in terms of Stirling numbers of the second kind and some relatively simple determinants, as seen in Proposition \ref{23-1prop}.

In the third section, we provide a comparable treatment for patterns of type (1,2).  Note that the cases 1-23 and 1-32 may be dealt with quickly, with the patterns 2-31 and 2-13 seen to be logically equivalent to the corresponding classical pattern of length three whose letters appear in the same order, which were considered in \cite{MSder}.  This leaves the patterns 3-12 and 3-21, which are treated in subsections \ref{s3-12} and \ref{s3-21}.  We make use of a pair of auxiliary statistics which track the first and last letters of the final cycle within a derangement, expressed in standard cycle form, to aid in writing a system of recurrences enumerating the members or $\mathcal{D}_{3\text{-}12}(n)$ or $\mathcal{D}_{3\text{-}21}(n)$.  In the case of 3-21, we must also consider an extra array which enumerates only those members whose final cycle has length at least three.  These recurrences then lead to a functional equation for the corresponding generating function in the case 3-12 and to a system of functional equations in the case 3-21; see Lemmas \ref{3-12lem2} and \ref{3-21lem2} below. Though they apparently cannot be solved explicitly, these functional equations do lead to recurrences for the generating functions which enumerate the members of either $\mathcal{D}_{3\text{-}12}(n)$ or $\mathcal{D}_{3\text{-}21}(n)$ having a fixed number $m$ of cycles for all $m \geq 1$.  Our results concerning the avoidance of patterns of type (2,1) or (1,2) by flattened derangements are summarized below in Table \ref{tab1}.

\begin{table}[htp]
\begin{center}
\begin{tabular}{l|l|l}
$\tau$&$\{d_{\tau}(n)\}_{n=2}^{10}$&Reference\\\hline\hline
$12\text{-}3$&1, 1, 3, 8, 27, 103, 436, 2025, 10207& Corollary \ref{12-3cor}\\\hline
$13\text{-}2$, $1\text{-}32$&1, 1, 2, 3, 5, 8, 13, 21, 34&Proposition \ref{prop13-2}\\\hline
$21\text{-}3$&1, 2, 7, 25, 101, 447, 2152, 11170, 62086&Theorem \ref{21-3th1}\\\hline
$23\text{-}1$, $32\text{-}1$&1, 2, 8, 32, 151, 784, 4467, 27568, 182820&Theorem \ref{23-1thm}\\\hline
$31\text{-}2$&1, 2, 7, 23, 80, 283, 1018, 3705, 13611&Theorem \ref{31-2thm}\\\hline
$1\text{-}23$&1, 1, 1, 1, 1, 1, 1, 1, 1&Trivial\\\hline
$2\text{-}13$&1, 2, 7, 23, 80, 283, 1018, 3705, 13611&\cite[Theorem~2.1]{MSder}\\\hline
$2\text{-}31$&1, 2, 8, 30, 124, 530, 2341, 10584, 48761&\cite[Theorem~2.4]{MSder}\\\hline
$3\text{-}12$&1, 2, 7, 25, 101, 444, 2116, 10849, 59518& Theorem \ref{3-12th1}\\\hline
$3\text{-}21$&1, 2, 8, 31, 139, 673, 3521, 19690, 117026&Theorem \ref{3-21th1}\\\hline
\end{tabular}
\caption{Values of $d_{\tau}(n)$ for $2\leq n \leq 10$, where $\tau$ is of type (2,1) or (1,2)}\label{tab1}
\end{center}
\end{table}

\section{Avoiding a vincular pattern of type (2,1)}

Throughout this section, we will denote $\text{flat}(\pi)$ for $\pi \in \mathcal{D}(n)$ by $\pi'$.
We first consider the case of avoiding 13-2, which can be dealt with quickly.  Let $f_n(y)=\sum_{k=0}^{\lfloor n/2 \rfloor}\binom{n-k}{k}y^k$ for $n \geq0$ denote the $n$-th standard Fibonacci polynomial.  Note the $f_n(y)$ reduces when $y=1$ to the $n$-th Fibonacci number $f_n$, indexed so that $f_0=f_1=1$.

\begin{proposition}\label{prop13-2}
If $n \geq 2$, then $d_{13\text{-}2}(n;y)=yf_{n-2}(y)$.
\end{proposition}
\begin{proof}
Let $\pi \in \mathcal{D}_{13\text{-}2}(n)$, where $n\geq 3$.  Note that the second letter of $\pi'$ must be 2 in order to avoid an occurrence of 13-2.  This in turn implies 3 must be the third letter, and so on, whence $\pi'=12\cdots n$.  Suppose $\pi$ contains $k+1$ cycles for some $k \geq 0$ and let $S$ denote the set of cycle starters of $\pi$.  Note that $S$ always contains 1, never contains $n$ and contains no two consecutive elements.  Thus, there are $\binom{n-2-k}{k}$ choices for $S$, where $|S|=k+1$ for some $0 \leq k \leq \lfloor (n-2)/2 \rfloor$.  This implies
$$d_{13\text{-}2}(n;y)=\sum_{k=0}^{\lfloor\frac{n-2}{2} \rfloor}\binom{n-2-k}{k}y^{k+1}=yf_{n-2}(y),$$
as desired.
\end{proof}

\subsection{The pattern 31-2}\label{s31-2}

Let $a_n=a_n(y)$ denote the distribution $d_{31\text{-}2}(n;y)$ for $n \geq 2$. Given $i \in [2,n]$, let $a_{n,i}=a_{n,i}(y)$ be the restriction of $a_n$ to those members of $\mathcal{D}_{31\text{-}2}(n)$ whose flattened form has second letter $i$.  Note $a_n=\sum_{i=2}^na_{n,i}$ for $n\geq 2$, by the definitions, with $a_1=0$. The array $a_{n,i}$ is given recursively as follows.

\begin{lemma}\label{31-2lem1}
If $n \geq 4$, then
\begin{equation}\label{31-2lem1e1}
a_{n,i}=\delta_{i,3}\cdot ya_{n-2}+\sum_{j=i-1}^{n-1}a_{n-1,j}, \qquad 3 \leq i \leq n-1,
\end{equation}
with $a_{n,2}=a_{n-1}+ya_{n-2}$ for $n\geq 3$ and $a_{n,n}=y$ for $n\geq 2$.
\end{lemma}
\begin{proof}
Let $\mathcal{A}_n=\mathcal{D}_{31\text{-}2}(n)$ and $\mathcal{A}_{n,i}$ denote the subset of $\mathcal{A}_n$ enumerated by $a_{n,i}$.  Note that $\mathcal{A}_{n,n}$ for $n \geq 2$ consists of only the single-cycle derangement $(1n(n-1)\cdots 2)$ by the ordering of cycles, whence $a_{n,n}=y$.  If $n \geq 3$, then the members of $\mathcal{A}_{n,2}$ are synonymous with members of $\mathcal{A}_{n-2}$ or $\mathcal{A}_{n-1}$ depending on whether or not the 2-cycle $(12)$ occurs, which explains the formula for $a_{n,2}$.  So assume $n \geq 4$ and $3 \leq i \leq n-1$.  If $i>3$, then $\pi \in \mathcal{A}_{n,i}$ implies $\pi'$ must start $1,i,j$ for some $j\in [i+1,n]\cup\{i-1\}$, as $j<i-1$ would produce an occurrence of 31-2 in which the role of 2 is played by the letter $i-1$. Note $i>3$ implies $j>2$, and hence $j$ must also belong to the first cycle of $\pi$, since $j$ cannot start a cycle as $2$ lies to the right of $j$ in $\pi'$. We may then delete $i$ from $\pi$ and reduce each letter in $[i+1,n]$ by one to obtain a member of $\mathcal{A}_{n-1,j-1}$ if $j \in [i+1,n]$ or a member of $\mathcal{A}_{n-1,i-1}$ if $j=i-1$. Considering all possible $j$ implies \eqref{31-2lem1e1} if $4 \leq i \leq n-1$.

If $i=3$, then similar reasoning applies concerning the contribution of $\sum_{j=2}^{n-1}a_{n-1,j}$ towards $a_{n,3}$, where here we assume in cases when $j=2$ that $2$ lies in the first cycle of $\pi$.  It is also possible when $\pi'$ starts $1,3,2$ that the 2 starts a cycle.  In these cases, we get a contribution of $ya_{n-2}$ towards the weight, where the factor of $y$ accounts for the 2-cycle $(13)$, which may be ignored as it is extraneous concerning the avoidance of 31-2.  Combining this case with the prior one implies \eqref{31-2lem1e1} when $i=3$ and completes the proof.
\end{proof}

Define the generating function $A(x,v)=A(x,y,v)$ by
$$A(x,v)=\sum_{n\geq2}\left(\sum_{i=2}^na_{n,i}v^{i-2}\right)x^n.$$
Then $A(x,v)$ satisfies the following functional equation.

\begin{lemma}\label{31-2lem2}
We have
\begin{equation}\label{31-2lem2e1}
(1-v+xv^2)A(x,v)=x^2y(1-v)+x(1+xy-xyv^2)A(x,1).
\end{equation}
\end{lemma}
\begin{proof}
First note
\begin{align*}
&\sum_{n\geq4}x^n\sum_{i=3}^{n-1}a_{n-1,i-1}v^{i-2}=\sum_{n\geq3}x^{n+1}\sum_{i=2}^{n-1}a_{n,i}v^{i-1}=xv\left(A(x,v)-\sum_{n\geq2}a_{n,n}x^nv^{n-2}\right)\\
&=xv\left(A(x,v)-\frac{x^2y}{1-xv}\right)
\end{align*}
and
\begin{align*}
&\sum_{n\geq4}\sum_{i=3}^{n-1}\sum_{j=i}^{n-1}a_{n-1,j}x^nv^{i-2}=\sum_{n\geq4}\sum_{j=3}^{n-1}a_{n-1,j}x^n\sum_{i=3}^jv^{i-2}=\sum_{n\geq3}\sum_{j=3}^{n}a_{n,j}x^{n+1}\left(\frac{v-v^{j-1}}{1-v}\right)\\
&=\frac{xv}{1-v}\sum_{n\geq2}\sum_{j=2}^na_{n,j}x^n(1-v^{j-2})=\frac{xv}{1-v}\left(A(x,1)-A(x,v)\right).
\end{align*}
Multiplying both sides of \eqref{31-2lem1e1} by $x^nv^{i-2}$, summing over all $n \geq 4$ and $3 \leq i \leq n-1$, and considering the contributions from $a_{n,2}$ and $a_{n,n}$ then yields
\begin{align*}
&A(x,v)-x^2y-x^3y(1+v)\\
&=xv\left(A(x,v)-\frac{x^2y}{1-xv}\right)+\frac{xv}{1-v}\left(A(x,1)-A(x,v)\right)+yv\sum_{n\geq4}a_{n-2}x^n\\
&\quad+\sum_{n\geq4}(a_{n-1}+ya_{n-2})x^n+y\sum_{n\geq4}x^nv^{n-2}\\
&=xv\left(A(x,v)-\frac{x^2y}{1-xv}\right)+\frac{xv}{1-v}\left(A(x,1)-A(x,v)\right)+x^2yvA(x,1)\\
&\quad+x(1+xy)A(x,1)-x^3y+\frac{x^4yv^2}{1-xv}.
\end{align*}
The last equality may be rewritten as
$$\left(1+\frac{xv^2}{1-v}\right)A(x,v)=x^2y+\left(\frac{x}{1-v}+x^2y(1+v)\right)A(x,1),$$
which implies \eqref{31-2lem2e1}.
\end{proof}

Let $C=C(x)$ denote the Catalan number generating function $\sum_{n\geq0}C_nx^n=\frac{1-\sqrt{1-4x}}{2x}$.  One can find an explicit formula for $A(x,v)$ and also for the coefficients of $A(x,1)$ as follows.

\begin{theorem}\label{31-2thm}
The generating function enumerating members of $\mathcal{D}_{31\text{-}2}(n)$ for $n \geq 2$ jointly according to the number of cycles and the second letter in the flattened form (marked by $y$ and $v$, respectively) is given by
\begin{equation}\label{31-2thme1}
A(x,v)=\frac{x^2y(1-v)}{1-v+xv^2}+\frac{x^2y(1+xy-xyv^2)(C-1)}{(1-v+xv^2)(1-xy(C^2-1))}.
\end{equation}
Moreover, the number of members of $\mathcal{D}_{31\text{-}2}(n)$ with $m$ cycles, where $1 \leq m \leq \lfloor n/2 \rfloor$, is given by
$$\frac{2(m-1)}{n-m}\sum_{j=0}^{m-2}(-1)^j\binom{m-2}{j}\binom{2n-2j-3}{n-m-2}+\frac{1}{n-m}\sum_{j=0}^{m-1}(-1)^j\binom{m-1}{j}\binom{2n-2j-2}{n-m-1}.$$
\end{theorem}
\begin{proof}
To solve \eqref{31-2lem2e1}, we apply the kernel method and let $v=C(x)$. By the fact $xC^2=C-1$, this cancels out the left-hand side of \eqref{31-2lem2e1} and solving for $A(x,1)$ gives
\begin{equation}\label{31-2thme2}
A(x,1)=\frac{xy(C-1)}{1-xy(C^2-1)}.
\end{equation}
Substituting \eqref{31-2thme2} back into \eqref{31-2lem2e1} yields \eqref{31-2thme1}.

Note that the number of members of $\mathcal{D}_{31\text{-}2}(n)$ with $m$ cycles is $[x^ny^m]A(x,1)$ for $1 \leq m \leq \lfloor n/2 \rfloor$.  By \eqref{31-2thme2}, we have
\begin{align*}
[x^ny^m]A(x,1)&=[x^ny^m]\left(\sum_{r\geq1}(C-1)(C^2-1)^{r-1}(xy)^r\right)=[x^{n-m}]\left((C-1)(C^2-1)^{m-1}\right)\\
&=[x^{n-m}]\left(\sum_{j=0}^{m-1}(-1)^j\binom{m-1}{j}C^{2m-2j-1}-\sum_{j=0}^{m-1}(-1)^j\binom{m-1}{j}C^{2m-2j-2}\right).
\end{align*}
Applying now the formula $[x^a]C^b=\frac{b}{a+b}\binom{2a+b-1}{a}$ for $a \geq 0$ and $b\geq 1$, see \cite[Equation~2.5.16]{Wilf}, we have
\begin{align*}
&[x^ny^m]A(x,1)\\
&=\sum_{j=0}^{m-1}(-1)^j\frac{2m-2j-1}{n+m-2j-1}\binom{m-1}{j}\binom{2n-2j-2}{n-m}-\sum_{j=0}^{m-2}(-1)^j\frac{2m-2j-2}{n+m-2j-2}\binom{m-1}{j}\binom{2n-2j-3}{n-m}\\
&=\sum_{j=0}^{m-1}(-1)^j\frac{2m-2j-1}{n-m}\binom{m-1}{j}\binom{2n-2j-2}{n-m-1}-\sum_{j=0}^{m-2}(-1)^j\frac{2m-2j-2}{n-m}\binom{m-1}{j}\binom{2n-2j-3}{n-m-1}\\
&=\frac{(-1)^{m-1}}{n-m}\binom{2n-2m}{n-m-1}\\
&\quad+\frac{1}{n-m}\sum_{j=0}^{m-2}(-1)^j\binom{m-1}{j}\left((2m-2j-1)\binom{2n-2j-2}{n-m-1}-(2m-2j-2)\binom{2n-2j-3}{n-m-1}\right)
\end{align*}
\begin{align*}
&=\frac{(-1)^{m-1}}{n-m}\binom{2n-2m}{n-m-1}+\frac{1}{n-m}\sum_{j=0}^{m-2}(-1)^j\binom{m-1}{j}\left((2m-2j-2)\binom{2n-2j-3}{n-m-2}+\binom{2n-2j-2}{n-m-1}\right)\\
&=\frac{(-1)^{m-1}}{n-m}\binom{2n-2m}{n-m-1}+\frac{2}{n-m}\sum_{j=0}^{m-2}(-1)^j(m-j-1)\binom{m-1}{j}\binom{2n-2j-3}{n-m-2}\\
&\quad+\frac{1}{n-m}\sum_{j=0}^{m-2}(-1)^j\binom{m-1}{j}\binom{2n-2j-2}{n-m-1}\\
&=\frac{2(m-1)}{n-m}\sum_{j=0}^{m-2}(-1)^j\binom{m-2}{j}\binom{2n-2j-3}{n-m-2}+\frac{1}{n-m}\sum_{j=0}^{m-1}(-1)^j\binom{m-1}{j}\binom{2n-2j-2}{n-m-1},
\end{align*}
as desired.
\end{proof}

\noindent{\bf Remark:} Summing the formula given in the second statement of Theorem \ref{31-2thm} over $1 \leq m \leq \lfloor n/2 \rfloor$ yields an expression for $d_{31\text{-}2}(n)$ for $n \geq 2$  in terms of binomial coefficients.  Using the fact $xC^2=C-1$, one can show
$$A(x,1)=\frac{x(C-1)}{1-x(C^2-1)}=\frac{x^2(2+x-(1+x)C)}{1-3x-4x^2-x^3},$$
and hence $d_{31\text{-1}2}(n)$ may also be expressed as a convolution involving the Catalan numbers and the entries of sequence A122600 in \cite{Sloane}, namely, the coefficients of $x^n$ in the expansion of $\frac{1}{1-3x-4x^2-x^3}$. \medskip

A permutation of $[2m]$ in which each (disjoint) cycle is of size two is known as a (perfect) \emph{matching} of $[2m]$.  By finding the coefficient of $[y^m]$ in $d_{\tau}(2m;y)$ for $m \geq 1$, equivalently the coefficient of $x^{2m}y^m$ in the corresponding generating function, one obtains the number of $\tau$-avoiding matchings of $[2m]$ in the flattened sense.  One may verify that a necessary and sufficient condition for a matching $\pi$ of $[2m]$ to avoid 31-2 is for $(\pi_{2i},\pi_{2i+1})$ to equal either $(2i,2i+1)$ or $(2i+1,2i)$ for all $i \in [m-1]$ within $\pi'=\pi_1\cdots\pi_{2m}$.  Thus, there are $2^{m-1}$ possibilities for $\pi'$, and hence the same number of 31-2 avoiding matchings of $[2m]$.  Equating this result with the $n=2m$ case of the second statement in Theorem \ref{31-2thm} yields the following somewhat curious alternating sum binomial identity, which we were unable to find in the literature.

\begin{corollary}\label{3-12cor1}
If $m \geq 0$, then
\begin{equation}\label{3-12cor1e1}
2^{m}=\frac{2m}{m+1}\sum_{j=0}^{m-1}(-1)^j\binom{m-1}{j}\binom{4m-2j+1}{m-1}+\frac{1}{m+1}\sum_{j=0}^{m}(-1)^j\binom{m}{j}\binom{4m-2j+2}{m}.
\end{equation}
\end{corollary}

\subsection{21-3 and 12-3}\label{s21-3}

Let $u_n=d_{21\text{-}3}(n;y)$ and let $u_{n,i}$ denote the restriction of $u_n$ to members of $\mathcal{D}_{21\text{-}3}(n)$ whose second letter in the flattened form is $i$.

\begin{lemma}\label{21-3lem1}
If $n \geq 4$, then
\begin{equation}\label{21-3lem1}
u_{n,i}=\sum_{j=i}^{n-1}u_{n-1,j}, \qquad 3 \leq i \leq n-1,
\end{equation}
with $u_{n,2}=u_{n,n}=u_{n-1}+yu_{n-2}$ for $n \geq 3$ and $u_{2,2}=y$.
\end{lemma}
\begin{proof}
Let $\mathcal{U}_n=\mathcal{D}_{21\text{-}3}(n)$ and $\mathcal{U}_{n,i}$ be the subset of $\mathcal{U}_n$ enumerated by $u_{n,i}$ for $2 \leq i \leq n$.  Then members of $\mathcal{U}_{n,2}$ or $\mathcal{U}_{n,n}$ are synonymous with members of $\mathcal{U}_{n-2}$ or $\mathcal{U}_{n-1}$ depending on whether or not the first cycle has length two, which implies the formula for $u_{n,2}$ and $u_{n,n}$.  If $\pi \in \mathcal{U}_{n,i}$, where $3 \leq i \leq n-1$, then the third letter in $\pi'$ must be $j+1$ for some $j \in [i,n-1]$, for otherwise, there would be an occurrence of 21-3 in which the role of 3 is played by  $i+1$.  Further, the first cycle of $\pi$ must have length at least three in this case since $j+1>i>2$.  One may thus delete $i$ from $\pi$ and reduce the letters in $[i+1,n]$ by one, and the resulting derangement is seen to belong to $\mathcal{U}_{n-1,j}$ for each $j$.  Considering all possible $j$ then yields \eqref{21-3lem1} and completes the proof.
\end{proof}

Define the generating function $U(x,v)=U(x,y,v)$ by
$$U(x,v)=\sum_{n\geq2}\left(\sum_{i=2}^nu_{n,i}v^{i-2}\right)x^n.$$

\begin{theorem}\label{21-3th1}
The generating function enumerating members of $\mathcal{D}_{21\text{-}3}(n)$ for $n \geq 2$ according to the number of cycles (marked by $y$) is given by
\begin{equation}\label{21-3th1e1}
U(x,1)=x^2y\sum_{j\geq0}\frac{x^j\prod_{i=1}^j(1-ix+xy)}{\prod_{i=0}^j(1-(2i+1)x+i(i+1)x^2-x^2y)}.
\end{equation}
\end{theorem}
\begin{proof}
Let $U_n(v)=\sum_{i=2}^nu_{n,i}v^{i-2}$ for $n \geq 2$ and note $u_n=U_n(1)$.  By \eqref{21-3lem1}, along with the formula for $u_{n,2}$ and $u_{n,n}$, we have
\begin{equation}\label{21-3th1e2}
U_n(v)=(U_{n-1}(1)+yU_{n-2}(1))(1+v^{n-2})+\frac{v}{1-v}(U_{n-1}(1)-U_{n-1}(v)), \qquad n \geq 4,
\end{equation}
with $U_2(v)=y$ and $U_3(v)=y+yv$.  Multiplying both sides of \eqref{21-3th1e2} by $x^n$, and summing over all $n \geq 4$, yields
\begin{align*}
&U(x,v)-x^2y-x^3y(1+v)\\
&=x(1+xy)U(x,1)-x^3y+\frac{x}{v}\left((1+xyv)U(xv,1)-(xv)^2y\right)+\frac{xv}{1-v}(U(x,1)-U(x,v)),
\end{align*}
which may be rewritten as
\begin{equation}\label{21-3th1e3}
\left(1+\frac{xv}{1-v}\right)U(x,v)=x^2y+\frac{x}{v}(1+xyv)U(xv,1)+x\left(xy+\frac{1}{1-v}\right)U(x,1).
\end{equation}
Applying now the kernel method, and letting $v=\frac{1}{1-x}$ in \eqref{21-3th1e3}, gives the functional equation
$$U(x,1)=\frac{x^2y}{1-x-x^2y}+\frac{x(1-x+xy)}{1-x-x^2y}U\left(\frac{x}{1-x},1\right),$$
which may be solved by iteration to give \eqref{21-3th1e1}.
\end{proof}

\noindent{\bf Remark:} Substituting \eqref{21-3th1e1} into \eqref{21-3th1e3} yields an explicit formula for $U(x,v)$, if desired. As a partial confirmation of \eqref{21-3th1e1}, we may use it to compute the number of 21-3 matchings of $[2m]$ for $m \geq 1$ as follows:
\begin{align*}
[x^{2m}y^m]U(x,1)&=\sum_{j=0}^{m-1}[x^{2m-2j-2}y^{m-j-1}]\left(\frac{1}{\prod_{i=0}^j(1-(2i+1)x+i(i+1)x^2-x^2y)}\right)\\
&=\sum_{j=0}^{m-1}[z^{m-j-1}]\left(\frac{1}{(1-z)^{j+1}}\right)=\sum_{j=0}^{m-1}\binom{m-1}{j}=2^{m-1}.
\end{align*}
This formula may also be realized directly by observing that in order for a matching $\pi$ of $[2m]$ to avoid 21-3, a necessary and sufficient condition is that $\pi_{2i}$ for each $i \in [m-1]$ within its flattened form $\pi'=\pi_1\cdots\pi_{2m}$ must either be the smallest or largest letter not occurring amongst $\pi_1\cdots\pi_{2i-1}$.  Then there are $2^{m-1}$ possibilities for $\pi'$, and hence also for $\pi$, as all of the letters $\pi_{2i+1}$ for $i \in [m-1]$ are uniquely determined once the choices regarding the sizes of $\pi_{2i}$ have been made. \medskip

We now seek a generating function formula enumerating the members of $\mathcal{W}_n=\mathcal{D}_{12\text{-}3}(n)$. First note that the second letter of $\pi'$ for $\pi \in \mathcal{W}_n$ where $n \geq 3$ must be $n$, for otherwise there would be an occurrence of 12-3 in $\pi'$ in which the role of the 3 is played by $n$.  This observation prompts one to consider a refinement of $\mathcal{W}_n$ as follows.  Given $n \geq 3$ and $2 \leq i \leq n-1$, let $\mathcal{W}_{n,i}$ denote the subset of $\mathcal{W}_n$ consisting of those members $\pi$ for which $\pi'$ starts $1,n,i$.  Let $w_n=d_{12\text{-}3}(n;y)$ and $w_{n,i}$ be the restriction of $w_n$ to $\mathcal{W}_{n,i}$ for $2 \leq i \leq n-1$.  For example, if $n=5$ and $i=3$, then $w_{5,3}=y^2+2y$, the enumerated derangements being $(153)(24)$, $(15324)$ and $(15342)$.

The array $w_{n,i}$ is given recursively as follows.

\begin{lemma}\label{12-3lem1}
If $n \geq 5$, then
\begin{equation}\label{12-3lem1e1}
w_{n,i}=w_{n-2}+\sum_{j=2}^{i-1}w_{n-1,j}, \qquad 3 \leq i \leq n-2,
\end{equation}
with $w_{n,2}=(1+y)w_{n-2}$, $w_{n,n-1}=w_{n-1}$ and initial conditions $w_{3,2}=w_{4,3}=y$, $w_{4,2}=y^2+y$.
\end{lemma}
\begin{proof}
The initial conditions for $n=3,4$ follow from observing $$\mathcal{W}_3=\{(132)\} \quad\text{and}\quad \mathcal{W}_4=\{(14)(23),\,(1423),\,(1432)\},$$ so assume $n \geq 5$.  Note that members $\pi \in \mathcal{W}_{n,2}$ in which the transposition $(1n)$ occurs have weight $yw_{n-2}$, whereas those $\pi$ in which it doesn't occur have weight $w_{n-2}$.  To realize the latter case, note that the first cycle of such $\pi$ has length at least four, as the fourth letter in $\pi'$ is $n-1 \geq 4$, with $n-1$ not starting a cycle as 3 occurs to the right of it in $\pi'$.  Combining these two cases then implies the formula for $w_{n,2}$.  Note that members of $\mathcal{W}_{n,n-1}$ and $\mathcal{W}_{n-1}$ are synonymous, whence $w_{n,n-1}=w_{n-1}$, upon deleting $n$ from the first cycle in members of the former set, which can be done as $n \geq 5$ implies the first cycle has length at least three.

Now suppose $3 \leq i \leq n-2$ and let $\pi \in \mathcal{W}_{n,i}$.  First assume that the fourth letter in $\pi'$ is $n-1$.  Then the first cycle of $\pi$ in this case is of length at least four since 2 occurs to the right of $i$ and $n-1$ in $\pi'$. Hence, we may delete both 1 and $n$ from the first cycle of $\pi$ and treat $i$ as the new ``smallest'' letter in the resulting derangement, thereby obtaining a contribution of $w_{n-2}$ towards $w_{n,i}$.  Otherwise, the fourth letter $j$ in $\pi'$ must belong to $[2,i-1]$ in order to avoid 12-3.  In this case, we may delete the extraneous letter $i$, which does not start a cycle of $\pi$ as 2 occurs to the right of $i$ in $\pi'$, and then reduce each letter in $[i+1,n]$ by one.  This is seen to produce an arbitrary member of $\mathcal{W}_{n-1,j}$ for each $j \in [2,i-1]$.  Considering all possible $j$ then yields a contribution of $\sum_{j=2}^{i-1}w_{n-1,j}$ towards  $w_{n,i}$.  Combining with the previous case implies \eqref{12-3lem1e1} and completes the proof.
\end{proof}

Define the generating function $W(x,v)$ by
$$W(x,v)=\sum_{n\geq3}\left(\sum_{i=2}^{n-1}w_{n,i}v^{n-1-i}\right)x^n.$$
Then $W(x,1)$ is given explicitly as follows.

\begin{theorem}\label{12-3th1}
The generating function enumerating members of $\mathcal{D}_{12\text{-}3}(n)$ for $n \geq 3$ according to the number of cycles is given by
\begin{equation}\label{12-3th1e1}
W(x,1)=x^3y\sum_{j\geq0}\frac{x^j\prod_{i=0}^j(1-ix+xy)}{(1-(j+1)x)\prod_{i=0}^j(1-ix)^2}.
\end{equation}
\end{theorem}
\begin{proof}
Let $W_n(v)=\sum_{i=2}^{n-1}w_{n,i}v^{n-1-i}$ for $n \geq 3$.  By Lemma \ref{12-3lem1}, we have
\begin{equation}\label{12-3th1e2}
W_n(v)=W_{n-1}(1)+yW_{n-2}(1)v^{n-3}
+W_{n-2}(1)\frac{v-v^{n-2}}{1-v}+\frac{v}{1-v}(W_{n-1}(1)-W_{n-1}(v)), \quad n \geq 5,
\end{equation}
with $W_3(v)=y$ and $W_4(v)=y+y(1+y)v$.  Multiplying both sides of \eqref{12-3th1e2} by $x^n$, and summing over all $n \geq 5$, we obtain
\begin{align*}
&W(x,v)-x^3(1+x)y-x^4y(1+y)v\\
&=x(W(x,1)-x^3y)+\frac{x^2y}{v}W(xv,1)+\frac{x^2}{1-v}(vW(x,1)-W(xv,1))+\frac{xv}{1-v}(W(x,1)-W(x,v)),
\end{align*}
which implies
\begin{equation}\label{12-3th1e3}
\left(1+\frac{xv}{1-v}\right)W(x,v)=x^3y(1+x(1+y)v)+\frac{x(1+xv)}{1-v}W(x,1)+\frac{x^2(y-v-yv)}{v(1-v)}W(xv,1).
\end{equation}
Letting $v=\frac{1}{1-x}$ in \eqref{12-3th1e3} gives
$$W(x,1)=\frac{x^3y(1+xy)}{1-x}+x(1-x)(1+xy)W\left(\frac{x}{1-x},1\right),$$
which may be iterated to obtain \eqref{12-3th1e1}.
\end{proof}

Let $S(n,k)$ denote the Stirling number of the second kind and $B_n=\sum_{k=0}^{n}S(n,k)$  the $n$-th Bell number.  We have the following explicit formula for $d_{12\text{-}3}(n)$ in terms of Stirling and Bell numbers.

\begin{corollary}\label{12-3cor}
If $n \geq 3$, then
\begin{equation}\label{12-3core1}
d_{12\text{-}3}(n)=B_{n-2}+\sum_{j=1}^{n-3}\sum_{r=j}^{n-3}jS(r,j)(j-1)^{n-r-3}.
\end{equation}
\end{corollary}
\begin{proof}
The formula is clear for $n=3$, as $d_{12\text{-}3}(3)=1$, so we may assume $n \geq 4$.  By \eqref{12-3th1e1}, we have
$$W(x,1)\mid_{y=1}=x^3\sum_{j\geq0}\frac{x^j(1+x)}{(1-jx)\prod_{i=1}^{j+1}(1-ix)}=x^2(1+x)\sum_{j\geq1}\frac{x^j}{(1-(j-1)x)\prod_{i=1}^j(1-ix)}.$$
Using the well-known fact $\frac{x^j}{\prod_{i=1}^j(1-ix)}=\sum_{r \geq j}S(r,j)x^r$ for $j \geq 1$ (see, e.g., \cite[p.~34]{Stan}), we then obtain
\begin{align*}
d_{12\text{-}3}(n)&=[x^n]W(x,1)\mid_{y=1}\\
&=[x^{n-2}]\left(\sum_{j\geq1}\frac{x^j}{\prod_{i=1}^j(1-ix)}\cdot\frac{1}{1-(j-1)x}\right)+[x^{n-3}]\left(\sum_{j\geq1}\frac{x^j}{\prod_{i=1}^j(1-ix)}\cdot\frac{1}{1-(j-1)x}\right)\\
&=\sum_{j=1}^{n-2}\sum_{r=j}^{n-2}S(r,j)(j-1)^{n-r-2}+\sum_{j=1}^{n-3}\sum_{r=j}^{n-3}S(r,j)(j-1)^{n-r-3}\\
&=\sum_{j=1}^{n-2}S(n-2,j)+\sum_{j=1}^{n-3}\sum_{r=j}^{n-3}S(r,j)\left((j-1)^{n-r-2}+(j-1)^{n-r-3}\right)\\
&=B_{n-2}+\sum_{j=1}^{n-3}\sum_{r=j}^{n-3}jS(r,j)(j-1)^{n-r-3},
\end{align*}
as desired.
\end{proof}

\subsection{23-1 and 32-1}\label{s23-1}

Note first the equivalence $23\text{-}1\approx32\text{-}1$ in the flattened sense for derangements, which follows from restricting the bijection given in \cite{MSW2} demonstrating the same equivalence for permutations in general.  Furthermore, this bijection is seen to preserve the cycle structure and hence, in particular, we have $d_{23\text{-}1}(n;y)=d_{32\text{-}1}(n;y)$.

We thus focus only upon enumerating the members of $\mathcal{E}_n=\mathcal{D}_{23\text{-}1}(n)$.  Let $d_{23\text{-}1}(n;y)$ be denoted by $e_n=e_n(y)$ for $n \geq 2$. The sequence $e_n$ is given recursively as follows.

\begin{lemma}\label{23-1lem1}
If $n \geq 3$, then
\begin{equation}\label{23-1lem1e1}
e_n=y+e_{n-1}+\sum_{i=1}^{n-3}\left(\binom{n-2}{i}+y\binom{n-1}{i}\right)e_{n-i-1},
\end{equation}
with $e_2=y$.
\end{lemma}
\begin{proof}
Given $\pi \in \mathcal{E}_n$ where $n \geq 3$, let $\alpha$ denote the subset of the letters in $[3,n]$ occurring between 1 and 2 in $\pi'$.  Note that the letters within $\alpha$ must be decreasing in order to avoid 23-1, but otherwise do not affect the avoidance or containment of 23-1 in the remaining section of $\pi'$.  We consider several cases on $\alpha$ as follows, where $i$ denotes the length of $\alpha$ in each case.  If the first cycle of $\pi$ comprises the letters in $1\alpha$, then deleting the first cycle is seen to result in a member of $\mathcal{E}_{n-i-1}$ on the alphabet $[2,n]-\alpha$.  Thus, the weight of such $\pi$ is given by $y\binom{n-2}{i}e_{n-i-1}$ for each $i$, where the factor of $y$ accounts for the first cycle.  Note $1 \leq i \leq n-3$ in this case since at least one letter must occur between 1 and 2 as well as to the right of 2 in $\pi'$.  Considering all possible $i$ then yields a contribution towards the overall weight $e_n$ of $y\sum_{i=1}^{n-3}\binom{n-2}{i}e_{n-i-1}$.

Now suppose that the first cycle of $\pi$ is given by $(1\alpha 2)$, with $\pi$ containing at least two cycles.  The we have $0 \leq i \leq n-4$ in this case, since at least two letters must occur to the right of 2 in $\pi'$ with $\alpha$ allowed to be empty, and hence we get a contribution of $y\sum_{i=0}^{n-4}\binom{n-2}{i}e_{n-i-2}$ from such $\pi$.  On the other hand, if the first cycle of $\pi$ has the form $(1\alpha2\beta)$, where $\beta$ is nonempty, then one may regard the section $2\beta$ as the first cycle within a member of $\mathcal{E}_{n-i-1}$, expressed using the alphabet $[2,n]-\alpha$.  Considering all possible $i$ then gives a contribution towards $e_n$ of $\sum_{i=0}^{n-3}\binom{n-2}{i}e_{n-i-1}$.  Note that no extra factor of $y$ is required in this case as the first cycle of $\pi$ is accounted for in $e_{n-i-1}$.  Finally, if $\pi$ consists of a single cycle ending in 2, then there is only one possibility, namely, $(1n(n-1)\cdots 2)$, which is of weight $y$.  Combining all of the preceding cases gives
\begin{equation}\label{23-1lem1e2}
e_n=y+y\sum_{i=1}^{n-3}\binom{n-2}{i}e_{n-i-1}+y\sum_{i=0}^{n-4}\binom{n-2}{i}e_{n-i-2}+\sum_{i=0}^{n-3}\binom{n-2}{i}e_{n-i-1}.
\end{equation}
Rewriting \eqref{23-1lem1e2}, using the fact $\binom{n-1}{i}=\binom{n-2}{i}+\binom{n-2}{i-1}$, gives \eqref{23-1lem1e1}.
\end{proof}

Due to the presence of the binomial coefficients in the recurrence for $e_n(y)$, it is more convenient to consider the exponential generating function $E(x,y)=\sum_{n\geq2}e_n(y)\frac{x^n}{n!}$.  Multiplying both sides of \eqref{23-1lem1e2} by $\frac{x^{n-2}}{(n-2)!}$, and summing over all $n \geq 3$, gives
\begin{align*}
\frac{\partial^2}{\partial x^2}E(x,y)-y&=\sum_{n\geq3}e_n\frac{x^{n-2}}{(n-2)!}\\
&=y\sum_{n\geq3}\frac{x^{n-2}}{(n-2)!}+\sum_{n\geq3}e_{n-1}\frac{x^{n-2}}{(n-2)!}+(y+1)\sum_{n\geq4}\frac{x^{n-2}}{(n-2)!}\sum_{i=1}^{n-3}\binom{n-2}{i}e_{n-i-1}\\
&\quad+y\sum_{n\geq4}\frac{x^{n-2}}{(n-2)!}\sum_{i=0}^{n-4}\binom{n-2}{i}e_{n-i-2}\\
&=y(e^x-1)+\frac{\partial}{\partial x}E(x,y)+(y+1)(e^x-1)\frac{\partial}{\partial x}E(x,y)+ye^xE(x,y),
\end{align*}
which implies $E(x,y)$ satisfies the following second-order differential equation.

\begin{lemma}\label{23-1lem2}
We have
\begin{equation}\label{23-1lem2e1}
\frac{\partial^2}{\partial x^2}E(x,y)=ye^x(1+E(x,y))+((y+1)e^x-y)\frac{\partial}{\partial x}E(x,y).
\end{equation}
\end{lemma}

Though an elementary closed-form solution cannot apparently be found for \eqref{23-1lem2e1}, we nonetheless have the following series expression for $E(x,y)$.

\begin{theorem}\label{23-1thm}
The exponential generating function enumerating members of $\mathcal{D}_{23\text{-}1}(n)$ for $n \geq 2$ according to the number of cycles is given by
\begin{equation}\label{23-1thme1}
E(x,y)=\sum_{k\geq2}\left[\frac{y(e^x-1)^k}{k!}\prod_{j=3}^k\left(2-j+\frac{j-2+(j-1)y}{3-j+\frac{j-3+(j-2)y}{-3+\frac{\ddots}{-2+\frac{2+3y}{-1}}}}\right)\right].
\end{equation}
Furthermore, the distribution for the number of cycles statistic on $\mathcal{D}_{23\text{-}1}(n)$ is given by
\begin{equation}\label{23-1thme2}
e_n(y)=y\sum_{k=2}^n\left[S(n,k)\prod_{j=3}^k\left(2-j+\frac{j-2+(j-1)y}{3-j+\frac{j-3+(j-2)y}{-3+\frac{\ddots}{-2+\frac{2+3y}{-1}}
}}\right)\right], \qquad n \geq2,
\end{equation}
where $S(n,k)$ denotes the Stirling number of the second kind.
\end{theorem}
\begin{proof}
Let $F(x,y)=E(x,y)+1$. By \eqref{23-1lem2e1}, we have
\begin{align*}
\frac{\partial^2}{\partial x^2}F(x,y)=ye^xF(x,y)+((y+1)e^x-y)\frac{\partial}{\partial x}F(x,y).
\end{align*}
Let $G$ be such that $F(x,y)=G(e^x-1,y)$.  Then generating function $G(t,y)$ satisfies
\begin{align}\label{23-1thme3}
(t+1)\frac{\partial^2}{\partial t^2}G(t,y) -t(y+1)\frac{\partial}{\partial t}G(t,y)-yG(t,y)&=0,
\end{align}
with $G(0,y)=1$ and $\frac{\partial}{\partial t}G(t,y)\mid_{t=0}=0$.
Define $G(t,y)=1+\sum_{n\geq2} a_nt^n$. By \eqref{23-1thme3}, the coefficients $a_n$ satisfy the recursion
\begin{equation}\label{23-1thme0}
a_{n+2}=\frac{y+(y+1)n}{(n+1)(n+2)}a_n-\frac{n}{n+2}a_{n+1}, \qquad n \geq 2,
\end{equation}
with $a_2=y/2$ and $a_3=-y/6$.

Define $b_n=\frac{a_n}{a_{n-1}}$. Then, by \eqref{23-1thme0}, we have
$$b_n=-\frac{n-2}{n}+\frac{n-2+(n-1)y}{n(n-1)b_{n-1}}, \qquad n \geq 4,$$
with $b_3=-1/3$, and hence
$$nb_n=
2-n+\frac{n-2+(n-1)y}{3-n+\frac{n-3+(n-2)y}{-3+\frac{\ddots}{-2+\frac{2+3y}{-1}}
}}.$$
Note $a_n=a_2b_3b_4\cdots b_n$, and thus since $a_2=y/2$, we have
$$a_n=\frac{y}{n!}\prod_{j=3}^n\left(2-j+\frac{j-2+(j-1)y}{3-j+\frac{j-3+(j-2)y}{-3+\frac{\ddots}{-2+\frac{2+3y}{-1}}
}}\right), \qquad n \geq 2.$$
Formula \eqref{23-1thme1} now follows from $E(x,y)=F(x,y)-1=G(e^x-1,y)-1$. The expression for $e_n(y)$ in \eqref{23-1thme2} then follows from extracting the coefficient of $\frac{x^n}{n!}$ on the right-hand side of \eqref{23-1thme1} and recalling the well-known fact $\frac{(e^x-1)^k}{k!}=\sum_{n \geq k}S(n,k)\frac{x^n}{n!}$, see, e.g., \cite[p.~34]{Stan}.
\end{proof}

Let $M_n$ for $n \geq 1$ denote the $n \times n$ matrix defined by

$$M_n=\left(\begin{array}{ccccccc}
-1       & -3     & 0  &0 &  \cdots       & 0      & 0 \\
y+\frac{2}{3}        & -2       & -4 &0  & \cdots        & 0      & 0 \\
0        & y+\frac{3}{4}      & -3   & -5   & \cdots        & 0      & 0 \\
\cdots     & \cdots &\cdots   & \cdots & \ddots & \cdots & \cdots \\
0   & \cdots   & \cdots  & 0 &y+\frac{n-1}{n} &-(n-1) & -(n+1) \\
0     & \cdots & \cdots   & \cdots & 0 &y+\frac{n}{n+1}   & -n
\end{array}\right)\!.$$
Then there is the following alternative expression for $e_n(y)$ to that given above in \eqref{23-1thme2} involving determinants of $M_n$.

\begin{proposition}\label{23-1prop}
If $n \geq 2$, then
\begin{equation}\label{23-1prope1}
e_n(y)=y\sum_{k=2}^nS(n,k)\det(M_{k-2}),
\end{equation}
where  $\det(M_0)=1$ and the matrix $M_{j}$ for $j \geq1$ is as defined above.
\end{proposition}
\begin{proof}
Let $\alpha_n=-\frac{n}{n+2}$ and $\beta_n=\frac{y+(y+1)n}{(n+1)(n+2)}$.  Then recurrence \eqref{23-1thme0} may be written as $a_{n+2}=\alpha_na_{n+1}+\beta_na_n$ for $n \geq 2$, with $a_2=\beta_0=y/2$ and $\alpha_3=\alpha_1\beta_0$.  This implies $a_{n+2}$ for $n \geq 1$ is given as a determinant by
$$a_{n+2}=\frac{y}{2}\left|
\begin{array}{lllllll}
\alpha_1&-1&0&0&\cdots&0\\
\beta_2&\alpha_2&-1&0&\cdots&0\\
0&\beta_3&\alpha_3&-1&\cdots&0\\
 &\ddots&\ddots&\ddots&\ddots&\\
0&\cdots&0&\beta_{n-1}&\alpha_{n-1}&-1\\
0&\cdots&0&0&\beta_{n}&\alpha_{n}\\
\end{array}
\right|,$$
and hence
$$a_{n+2}=\frac{y}{(n+2)!}\left|
\begin{array}{lllllll}
-1&-3&0&0&\cdots&0\\
y+\frac{2}{3}&-2&-4&0&\cdots&0\\
0&y+\frac{3}{4}&-3&-5&\cdots&0\\
 &\ddots&\ddots&\ddots&\ddots&\\
0&\cdots&0&y+\frac{n-1}{n}&-(n-1)&-(n+1)\\
0&\cdots&0&0&y+\frac{n}{n+1}&-n\\
\end{array}
\right|.$$
By \eqref{23-1thme2} and the definitions, we then have
$$e_n(y)=\sum_{k=2}^nS(n,k)k!a_k=y\sum_{k=2}^nS(n,k)\det(M_{k-2}),$$
as desired.
\end{proof}

\section{Patterns of type (1,2)}

In this section, we consider the problem of avoiding a vincular pattern of type (1,2) in the flattened sense.  We first remark that the case 1-23 is trivial in that a derangement $\pi$ which avoids 1-23 must have flattened form $1n(n-1)\cdots 2$, and hence only $(1n(n-1)\cdots 2)$ is possible, by the ordering of cycles.  Also, if $\pi \in \mathcal{D}(n)$ avoids 1-32, then $\pi'=1\cdots n$, and hence $d_{1\text{-}32}(n;y)=yf_{n-2}(y)$, upon reasoning as in the proof of Proposition \ref{prop13-2}.  In particular, we have $d_{1\text{-}32}(n)=f_{n-2}$ for all $n \geq 2$.

We now argue that 2-31 is logically equivalent to the classical patterns 2-3-1, with the avoiders of the latter having already been enumerated in \cite[Theorem~2.4]{MSder}.  Clearly, we have $\mathcal{D}_{2\text{-}3\text{-}1}(n)\subseteq \mathcal{D}_{2\text{-}31}(n)$.  To show the reverse inclusion, we argue that if $\pi \in \mathcal{D}(n)$ is such that its flattened form $\pi'=\pi_1\cdots \pi_n$ contains an occurrence of 2-3-1, then it must also contain 2-31. Suppose $\pi_i\pi_j\pi_k$ is such that $\pi_k<\pi_i<\pi_j$, where $1\leq i<j<k\leq n$.  Let $\ell_0$ denote the smallest index $\ell\in[j+1,k]$ such that $\pi_\ell<\pi_i$; note that $\ell_0$ exists since $\pi_k<\pi_i$.  Then each of $\pi_j,\pi_{j-1},\ldots,\pi_{\ell_0-1}$ must be greater than $\pi_i$, by the minimality of $\ell_0$.  In particular, we have that $\pi_i\pi_{\ell_0-1}\pi_{\ell_0}$ is an occurrence of 2-31 in $\pi'$. Similar reasoning shows that 2-13 and 2-1-3 are logically equivalent as well, with avoiders of the latter having been enumerated in \cite[Theorem~2.1]{MSder}.

The remaining two patterns of type (1,2), namely 3-12 and 3-21, are treated in the following two subsections.

\subsection{The case 3-12}\label{s3-12}

To determine a recurrence enumerating the members of $\mathcal{D}_\tau(n)$ in this and the subsequent case, we will consider a refinement of the distribution $d_\tau(n;y)$ as follows.  Let $z_n=d_{3\text{-}12}(n;y)$ for $n \geq 2$, with $z_1=0$, and let $z_{n,i,j}$ for $1 \leq i < j \leq n$ denote the restriction of $z_n$ to the members of $\mathcal{D}_{3\text{-}12}(n)$ whose last cycle (in standard cycle form) starts with $i$ and ends in $j$.
Note $z_n=\sum_{i=1}^{n-1}\sum_{j=i+1}^nz_{n,i,j}$, by the definitions.  For example, when $n=4$, we have  $z_{4,1,2}=z_{4,1,4}=2y$, $z_{4,1,3}=y$,  $z_{4,2,3}=0$, $z_{4,2,4}=z_{4,3,4}=y^2$, and hence $z_4=2y^2+5y$.

The array $z_{n,i,j}$ is given recursively as follows.

\begin{lemma}\label{3-12lem1}
If $n \geq 3$, then
\begin{equation}\label{3-12lem1e1}
z_{n,i,j}=\sum_{\ell=j}^{n-1}z_{n-1,i,\ell},  \qquad 1 \leq i <j \leq n-1,
\end{equation}
and
\begin{equation}\label{3-12lem1e2}
z_{n,i,n}=\sum_{\ell=i+1}^{n-1}z_{n-1,i,\ell}+y\sum_{\ell=1}^{i-1}z_{n-1}{\ell,i}, \qquad 1 \leq i \leq n-2,
\end{equation}
with $z_{n,n-1,n}=yz_{n-2}$ and the initial condition $z_{2,1,2}=y$.
\end{lemma}
\begin{proof}
We may assume $n \geq 3$, the initial condition for $n=2$ being clear.  Let $\mathcal{Z}_{n,i,j}$ denote the subset of $\mathcal{D}_{3\text{-}12}(n)$ whose members are enumerated by $z_{n,i,j}$ for $1 \leq i < j \leq n$. Let $\pi \in \mathcal{Z}_{n,i,j}$, with $\pi'=\pi_1\cdots \pi_n$, and first suppose $j<n$.  Note that we then must have $\pi_{n-1}=\ell+1$ for some $\ell \in [j,n-1]$, for otherwise there would be an occurrence of 3-12 as witnessed by the subsequence  $\pi_s\pi_{n-1}\pi_n$ of $\pi'$, where $s<n-1$ is such that $\pi_s=n$.  Deleting $j$, and then reducing each letter of $\pi$ in $[j+1,n]$ by one, is seen to yield an arbitrary member of $\mathcal{Z}_{n-1,i,\ell}$, as the first element of the last cycle is unchanged, with this cycle initially being of length at least three.  Considering all $\ell \in [j,n-1]$ then implies \eqref{3-12lem1e1}.

Now suppose $\pi \in \mathcal{Z}_{n,i,n}$, with $i \in [n-2]$.  If the last cycle of $\pi$ has length three or more, with $\pi_{n-1}=\ell$, then $\ell \in [i+1,n-1]$, by the ordering of cycles, and considering all possible $\ell$ yields a contribution towards the weight of $\sum_{\ell=i+1}^{n-1}z_{n-1,i,\ell}$.  On the other hand, if the final cycle of $\pi$ has length two, then we may delete the terminal cycle $(i n)$ and append $i$ to the penultimate cycle to obtain a derangement of length $n-1$, where $n \geq 3$ implies $i \geq 2$ in this case.  Indeed, all of the members of $\cup_{\ell=1}^{i-1}\mathcal{Z}_{n-1,\ell,i}$ arise uniquely in this fashion since $i<n-1$ implies each derangement in this set has terminal cycle length at least three.  Taking into account the deleted cycle $(i n)$, we get a contribution of $y\sum_{\ell=1}^{i-1}z_{n-1,\ell,i}$, and combining with the prior case yields \eqref{3-12lem1e2}.  Finally, members of the sets $\mathcal{Z}_{n,n-1,n}$ and $\mathcal{D}_{3\text{-}12}(n-2)$ where $n \geq 4$ are seen to be synonymous, upon deleting the terminal cycle containing $n-1$ and $n$ from members of the former.  This implies $z_{n,n-1,n}=yz_{n-2}$ and completes the proof.
\end{proof}

Let $Z_n(u,v)=\sum_{i=1}^{n-1}\sum_{j=i+1}^nz_{n,i,j}u^{n-1-i}v^{n-j}$ for $n \geq 2$ and $Z(x,u,v)=\sum_{n\geq2}Z_n(u,v)x^n$.  Then $Z(x,u,v)$ satisfies the following functional equation.

\begin{lemma}\label{3-12lem2}
We have
\begin{equation}\label{3-12lem2e1}
\left(1-\frac{xuv}{1-v}\right)Z(x,u,v)=\frac{x^2y}{1+xy}-\frac{xuv^2}{1-v}Z(x,uv,1)+xuZ(x,u,1)+xyZ(x,1,u)-\frac{x^2y^2}{1+xy}Z(x,1,1).
\end{equation}
\end{lemma}
\begin{proof}
Using \eqref{3-12lem1e1} and \eqref{3-12lem1e2}, one can show
\begin{align}
Z_n(u,v)&=\frac{uv}{1-v}\left(Z_{n-1}(u,v)-vZ_{n-1}(uv,1)\right)+uZ_{n-1}(u,1)\notag\\
&\quad+y\left(Z_{n-1}(1,u)+Z_{n-2}(1,1)-Z_{n-1}(1,0)\right), \qquad n \geq 3, \label{Znrec1}
\end{align}
with $Z_1(u,v)=0$ and $Z_2(u,v)=y$. By \eqref{Znrec1}, we then have
\begin{align}
\left(1-\frac{xuv}{1-v}\right)Z(x,u,v)&=x^2y-\frac{xuv^2}{1-v}Z(x,uv,1)+xuZ(x,u,1)\notag\\
&\quad+xy(Z(x,1,u)+xZ(x,1,1)-Z(x,1,0)).\label{Znfuneq1}
\end{align}

We may eliminate the $Z(x,1,0)$ term from \eqref{Znfuneq1} as follows.  First note that by \eqref{3-12lem1e2}, we have
\begin{align}
Z_n(1,0)&=\sum_{i=1}^{n-1}z_{n,i,n}=\sum_{i=1}^{n-2}\sum_{\ell=i+1}^{n-1}z_{n-1,i,\ell}+y\sum_{i=2}^{n-2}\sum_{\ell=1}^{i-1}z_{n-1,\ell,i}+yz_{n-2}\notag\\
&=Z_{n-1}(1,1)+yZ_{n-2}(1,1)+y\left(\sum_{\ell=1}^{n-2}\sum_{i=\ell+1}^{n-1}z_{n-1,\ell,i}-\sum_{\ell=1}^{n-2}z_{n-1,\ell,n-1}\right)\notag\\
&=Z_{n-1}(1,1)+yZ_{n-2}(1,1)+y(Z_{n-1}(1,1)-Z_{n-1}(1,0))\notag\\
&=(1+y)Z_{n-1}(1,1)+yZ_{n-2}(1,1)-yZ_{n-1}(1,0), \qquad n \geq 3.\label{Znfuneq2}
\end{align}
Multiplying both sides of \eqref{Znfuneq2} by $x^n$, summing over all $n \geq 3$ and solving for $Z(x,1,0)$ yields
\begin{equation}\label{znfuncren}
Z(x,1,0)=\frac{x^2y}{1+xy}+\frac{x(1+y+xy)}{1+xy}Z(x,1,1).
\end{equation}
Formula \eqref{3-12lem2e1} now follows from substituting \eqref{znfuncren} into \eqref{Znfuneq1} and simplifying.
\end{proof}

Though \eqref{3-12lem2e1} does not appear to have an explicit solution for $Z(x,u,v)$, it is still possible to obtain the following formula involving the generating function $Z_m(x,u,v)=[y^m]Z(x,u,v)$ for $m \geq 1$, where $$Z'_m(x,u)=(-x)^{m+1}+xZ_{m-1}(x,1,u)
-\sum_{j=2}^{m-1}(-x)^jZ_{m-j}(x,1,1), \qquad m \geq 2,$$ with $Z_1'(x,u)=x^2$.

\begin{theorem}\label{3-12th1}
Let $m \geq1$ be fixed. The generating function $Z_m(x,u,v)$ enumerating the members of $\mathcal{D}_{3\text{-}12}(n)$ for $n \geq 2m$ having exactly $m$ cycles according to the parameters tracking $n-1-i$ and $n-j$ (marked by $u$ and $v$) is given recursively by
\begin{align}\label{3-12th1e1}
Z_m(x,u,v)&=\frac{1-v}{1-v-xuv}Z'_m(x,u)-\frac{xuv^2}{1-v-xuv}\sum_{j\geq0}\frac{(xuv)^j}{\prod_{i=1}^{j+1}(1-ixuv)}Z'_m\left(x,\frac{uv}{1-(j+1)xuv}\right)\nonumber\\
&\quad+\frac{xu(1-v)}{1-v-xuv}\sum_{j\geq0}\frac{(xu)^j}{\prod_{i=1}^{j+1}(1-ixu)}Z'_m\left(x,\frac{u}{1-(j+1)xu}\right), \qquad m \geq1,
\end{align}
where $i$ and $j$ denote the first and last letters respectively of the final cycle in a member of $\mathcal{D}_{3\text{-}12}(n)$.
\end{theorem}
\begin{proof}
Extracting the coefficient of $y^m$ on both sides of \eqref{3-12lem2e1}, we have
\begin{align*}
\left(1-\frac{xuv}{1-v}\right)Z_m(x,u,v)&=(-x)^{m+1}
-\frac{xuv^2}{1-v}Z_m(x,uv,1)+xuZ_m(x,u,1)+xZ_{m-1}(x,1,u)\\
&\quad-\sum_{j=2}^{m-1}(-x)^jZ_{m-j}(x,1,1),
\end{align*}
which implies
\begin{align}\label{3-12th1e2}
\left(1-\frac{xuv}{u-v}\right)Z_m(x,u,v/u)&=-\frac{xv^2}{u-v}Z_m(x,v,1)+xuZ_m(x,u,1)+Z'_m(x,u), \qquad m \geq1,
\end{align}
where $Z'_m(x,u)$ is as stated, upon replacing $v$ with $v/u$.  Applying the kernel method, and taking $u=\frac{v}{1-xv}$ in \eqref{3-12th1e2}, we obtain
\begin{equation}\label{3-12th1e3}
Z_m(x,v,1)=\frac{xv}{(1-xv)^2}
Z_m\left(x,\frac{v}{1-xv},1\right)+\frac{1}{1-xv}Z'_m\left(x,\frac{v}{1-xv}\right).
\end{equation}
Iteration of \eqref{3-12th1e3}, where we may assume $|x|<1$, leads to
$$Z_m(x,v,1)=\sum_{j\geq0}\frac{x^jv^j}{\prod_{i=1}^{j+1}(1-ixv)}Z'_m\left(x,\frac{v}{1-(j+1)xv}\right),$$
from which \eqref{3-12th1e1} follows from \eqref{3-12th1e2}.
\end{proof}

Formula \eqref{3-12th1e1} provides a recursive procedure for finding $Z_m(x,u,v)$ for $m \geq1$, as $Z'_m(x,u)$ depends only upon the generating functions $Z_j(x,v,u)$ for $1 \leq j \leq m-1$. Note that the number of cycles parameter, in addition to the statistics marked by $u$ and $v$, is crucial in finding a procedure for recursively enumerating flattened derangements that avoid 3-12 (and also 3-21 below).

\subsection{The case 3-21}\label{s3-21}

Let $r_{n,i,j}$ be defined in analogy with $z_{n,i,j}$ above, but in conjunction with the pattern 3-21.  To aid in finding a recursion for $r_{n,i,j}$, let $s_{n,i,j}$ denote the restriction of $r_{n,i,j}$ to those members in which the final cycle has length at least three.  For example, when $n=4$, we have $r_{4,1,2}=y$, $r_{4,1,3}=r_{4,1,4}=2y$, $r_{4,2,3}=r_{4,2,4}=r_{4,3,4}=y^2$, with $s_{4,1,j}=r_{4,1,j}$ for $2 \leq j \leq 4$ and $s_{4,i,j}=0$ if $i=2,3$.

The arrays $r_{n,i,j}$ and $s_{n,i,j}$ satisfy the following system of intertwined recurrences, where $[S]$ equals
one or zero depending upon the truth or falsehood of the statement $S$.

\begin{lemma}\label{3-21lem1}
If $n \geq 3$, then
\begin{equation}\label{rn/snrec1}
r_{n,i,j}=[j<n]r_{n-1,i,n-1}+\sum_{k=i+1}^{j-1}r_{n-1,i,k}+y\sum_{\ell=1}^{i-1}s_{n-1,\ell,i}, \qquad 1 \leq i <j \leq n,
\end{equation}
and
\begin{equation}\label{rn/snrec2}
s_{n,i,j}=r_{n,i,j}-y\sum_{\ell=1}^{i-1}s_{n-1,\ell,i}, \qquad 1 \leq i < j \leq n,
\end{equation}
with $r_{2,1,2}=y$ and $s_{2,1,2}=0$.
\end{lemma}
\begin{proof}
The initial conditions when $n=2$ are clear, so assume $n \geq 3$.  Let $\mathcal{R}_{n,i,j}$ and $\mathcal{S}_{n,i,j}$ denote the subsets of $\mathcal{D}_{3\text{-}21}(n)$ whose members are enumerated by $r_{n,i,j}$ and $s_{n,i,j}$, respectively.  We first show the $j=n$ case of \eqref{rn/snrec1}, that is,
\begin{equation}\label{rn/snrec3}
r_{n,i,n}=\sum_{k=i+1}^{n-1}r_{n-1,i,k}+y\sum_{\ell=1}^{i-1}s_{n-1,\ell,i}, \qquad 1 \leq i \leq n-1.
\end{equation}
Let $\pi \in \mathcal{R}_{n,i,n}$, where $n \geq 3$ and $1 \leq i \leq n-1$.  If the final cycle of $\pi$ has length at least three, then deletion of $n$ implies that one gets a contribution of $\sum_{k=i+1}^{n-1}r_{n-1,i,k}$ towards $r_{n,i,n}$, upon considering the penultimate letter $k \in [i+1,n-1]$ of $\pi'$.  Otherwise, the final cycle of $\pi$ is the transposition $(in)$, in which case, we delete $n$ and add $i$ as the last letter to the penultimate cycle of $\pi$ (which exists as $n\geq 3$ implies $i \geq 2$ in this case).  This is seen to yield each member of $\cup_{\ell=1}^{i-1}\mathcal{S}_{n-1,\ell,i}$ in a unique manner, and hence we obtain a contribution of $y\sum_{\ell=1}^{i-1}s_{n-1,\ell,i}$, upon considering the first letter $\ell$ in the penultimate cycle of $\pi$.  Note that the extra factor of $y$ accounts for the deleted cycle $(in)$.  Combining this case with the prior implies \eqref{rn/snrec3}.

Now assume $\pi \in \mathcal{R}_{n,i,j}$, where $1 \leq i <j \leq n-1$. Let $k$ denote the penultimate letter of $\pi'$.  Note $k \in [i,j-1]\cup\{n\}$ in order to avoid 3-21.  If $k=n$, then deletion of $j$ results in a member of $\mathcal{R}_{n-1,i,n-1}$, as the final cycle of $\pi$ must contain at least three letters in this case, and hence we obtain a contribution of $r_{n-1,i,n-1}$ towards the weight. If $k \in [i+1,j-1]$, then deletion of $j$ yields $\sum_{k=i+1}^{j-1}r_{n-1,i,k}$, as again the final cycle must have length at least three.  If $k=i$, then the terminal cycle of $\pi$ is $(ij)$, and proceeding as in the comparable case above in the proof of \eqref{rn/snrec3} yields a contribution of $y\sum_{\ell=1}^{i-1}s_{n-1,\ell,i}$ towards $r_{n,i,j}$.  Combining each of the prior cases regarding $k$ then implies the $j<n$ case of \eqref{rn/snrec1} and finishes the proof of \eqref{rn/snrec1}.  Finally, in both of the preceding arguments wherein $j=n$ or $j<n$, respectively, it is seen that the weight of all members of $\mathcal{R}_{n,i,j}$ in which the final cycle has length two equals $y\sum_{\ell=1}^{i-1}s_{n-1,\ell,i}$.  Subtracting this quantity from $r_{n,i,j}$ then gives an expression for $s_{n,i,j}$, by the definitions, which implies \eqref{rn/snrec2} and completes the proof.
\end{proof}

Let $R_n(u,v)=\sum_{i=1}^{n-1}\sum_{j=i+1}^{n}r_{n,i,j}u^{n-1-i}v^{n-j}$ and $S_n(u,v)=\sum_{i=1}^{n-1}\sum_{j=i+1}^{n}s_{n,i,j}u^{n-1-i}v^{n-j}$ for $n \geq 2$.  Define the generating functions $R(x,u,v)=\sum_{n\geq 2}R_n(u,v)x^n$ and $S(x,u,v)=\sum_{n\geq 2}S_n(u,v)x^n$.

\begin{lemma}\label{3-21lem2}
We have
\begin{align}
\left(1+\frac{xuv}{1-v}\right)R(x,u,v)&=x^2y(1+xuv)+\frac{xu(1+xuv)}{1-v}R(x,u,1)-\frac{x^2u^2v^3}{1-v}R(x,uv,1)\notag\\
&\quad+\frac{xy(1+xuv)}{1-v}(S(x,1,u)-vS(x,1,uv))\label{3-21lem2e1}
\end{align}
and
\begin{equation}\label{3-21lem2e2}
S(x,u,v)=R(x,u,v)-x^2y-\frac{xy}{1-v}(S(x,1,u)-vS(x,1,uv)).
\end{equation}
\end{lemma}
\begin{proof}
Using \eqref{rn/snrec1} and \eqref{rn/snrec2}, one can show for $n \geq 3$,
\begin{align}
R_n(u,v)&=\frac{uv}{1-v}(R_{n-1}(u,0)-vR_{n-1}(uv,0))+\frac{u}{1-v}(R_{n-1}(u,1)-vR_{n-1}(u,v))\notag\\
&\quad+\frac{y}{1-v}(S_{n-1}(1,u)-vS_{n-1}(1,uv))\label{3-21lem2e3}
\end{align}
and
\begin{equation}\label{3-21lem2e4}
S_n(u,v)=R_n(u,v)-\frac{y}{1-v}(S_{n-1}(1,u)-vS_{n-1}(1,uv)),
\end{equation}
with $R_2(u,v)=y$ and $S_2(u,v)=0$.  Note that \eqref{3-21lem2e4} implies \eqref{3-21lem2e2}, whereas from \eqref{3-21lem2e3}, we get
\begin{align}
\left(1+\frac{xuv}{1-v}\right)R(x,u,v)&=x^2y+\frac{xu}{1-v}R(x,u,1)+\frac{xuv}{1-v}(R(x,u,0)-vR(x,uv,0))\notag\\
&\quad+\frac{xy}{1-v}(S(x,1,u)-vS(x,1,uv)).\label{3-21lem2e5}
\end{align}

To eliminate the terms in \eqref{3-21lem2e5} where $v=0$, first note that by \eqref{rn/snrec3}, we have
\begin{align*}
R_n(u,0)&=\sum_{i=1}^{n-1}r_{n,i,n}u^{n-1-i}=\sum_{i=1}^{n-2}u^{n-1-i}\sum_{k=i+1}^{n-1}r_{n-1,i,k}+y\sum_{i=2}^{n-1}u^{n-1-i}\sum_{\ell=1}^{i-1}s_{n-1,\ell,i}\\
&=\sum_{i=1}^{n-2}\sum_{k=i+1}^{n-1}r_{n-1,i,k}u^{n-1-i}+y\sum_{\ell=1}^{n-2}\sum_{i=\ell+1}^{n-1}s_{n-1,\ell,i}u^{n-1-i}\\
&=uR_{n-1}(u,1)+yS_{n-1}(1,u), \qquad n \geq 3,
\end{align*}
which implies
\begin{equation}\label{3-21lem2e6}
R(x,u,0)=x^2y+xuR(x,u,1)+xyS(x,1,u).
\end{equation}
Substituting \eqref{3-21lem2e6} into \eqref{3-21lem2e5}, and simplifying, yields \eqref{3-21lem2e1}.
\end{proof}

Let $R_m(x,u,v)=[y^m]R(x,u,v)$ and $S_m(x,u,v)=[y^m]S(x,u,v)$ for $m \geq 1$, with $R_0(x,u,v)=S_0(x,u,v)=0$.  Though it is apparently not possible to solve the system of functional equations in Lemma \ref{3-21lem2} explicitly, we have the following recursive formulas for $R_m(x,u,v)$ and $S_m(x,u,v)$.

\begin{theorem}\label{3-21th1}
Let $m \geq1$ be fixed. The generating function $R_m(x,u,v)$ enumerating the members of $\mathcal{D}_{3\text{-}21}(n)$ for $n \geq 2m$ having exactly $m$ cycles according to the parameters tracking $n-1-i$ and $n-j$ (marked by $u$ and $v$) is given recursively by
\begin{align}
R_m(x,u,v)&=\frac{x^2(1-v)(1+xuv)}{1-v+xuv}\cdot\delta_{m,1}+\frac{x(1+xuv)}{1-v+xuv}(S_{m-1}(x,1,u)-vS_{m-1}(x,1,uv))\notag\\ &\quad+\frac{x(1+xuv)}{1-v+xuv}\sum_{j\geq0}\frac{(xu)^j(1-jxu)}{\prod_{i=1}^{j+1}(1-ixu)}R'_m\left(x,\frac{u}{1-jxu}\right)\notag\\
&\quad-\frac{x^2uv^2}{1-v+xuv}\sum_{j\geq0}\frac{(xuv)^j(1-jxuv)}{\prod_{i=1}^{j+1}(1-ixuv)}R'_m\left(x,\frac{uv}{1-jxuv}\right),\qquad m \geq1, \label{3-21th1e1}
\end{align}
where $i$ and $j$ denote the first and last letters respectively of the final cycle in a member of $\mathcal{D}_{3\text{-}21}(n)$ and
$$R'_m(x,u)=x^2u\cdot\delta_{m,1}-(1-xu)S_{m-1}(x,1,u)+S_{m-1}\left(x,1,\frac{u}{1-xu}\right).$$
Furthermore, we have $S_m(x,u,v)=R_m(x,u,v)-x^2\cdot\delta_{m,1}-\frac{x}{1-v}(S_{m-1}(x,1,u)-vS_{m-1}(x,1,uv))$,
where $S_m(x,u,v)$ is the restriction of $R_m(x,u,v)$ to derangements whose final cycle is of length at least three.
\end{theorem}
\begin{proof}
The second statement follows from extracting coefficients of $y^m$ on both sides of \eqref{3-21lem2e2}, so we focus on establishing \eqref{3-21th1e1}.  Comparing coefficients of $y^m$ in \eqref{3-21lem2e1}, and replacing $v$ with $v/u$, gives
\begin{align}
\left(1+\frac{xuv}{u-v}\right)R_m(x,u,v/u)&=x^2(1+xv)\cdot\delta_{m,1}+\frac{xu^2(1+xv)}{u-v}R_m(x,u,1)-\frac{x^2v^3}{u-v}R_m(x,v,1)\notag\\
&\quad+\frac{xu(1+xv)}{u-v}(S_{m-1}(x,1,u)-(v/u)S_{m-1}(x,1,v)).\label{3-21th1e2}
\end{align}
Letting $v=\frac{u}{1-xu}$ in \eqref{3-21th1e2} implies
\begin{equation}\label{3-21th1e3}
R_m(x,u,1)=\frac{xu}{(1-xu)^2}R_m\left(x,\frac{u}{1-xu},1\right)+\frac{1}{u(1-xu)}R'_m(x,u),
\end{equation}
where $R'_m(x,u)$ is as stated.  Iteration of \eqref{3-21th1e3} leads to
\begin{equation}\label{3-21th1e4}
R_m(x,u,1)=\sum_{j\geq0}\frac{x^ju^{j-1}(1-jxu)}{\prod_{i=1}^{j+1}(1-ixu)}R'_m\left(x,\frac{u}{1-jxu}\right), \qquad m \geq 1.
\end{equation}
Substituting \eqref{3-21th1e4} into \eqref{3-21th1e2} implies \eqref{3-21th1e1} and completes the proof.
\end{proof}

\noindent{\bf Remark:} Note that when $m=1$, we have $Z'_1(x,u)=x^2$ and $R'_1(x,u)=x^2u$.  Thus, taking $m=1$ in Theorems \ref{3-12th1} and \ref{3-21th1} gives
$$Z_1(x,u,v)=\frac{x^2}{1-v-xuv}((1-v)Bell(xu)-vBell(xuv)+v)$$
and
$$R_1(x,u,v)=\frac{x^2}{1-v+xuv}((1+xuv)Bell(xu)-xuv^2Bell(xuv)-v),$$
where $Bell(x)=\sum_{j\geq0}\frac{x^j}{\prod_{i=1}^j(1-ix)}$ is the (ordinary) generating function of the Bell number sequence $B_n$ for $n\geq0$.  Let $\mathcal{S}_n(\tau)$ denote the set of permutations of length $n$ that avoid the pattern $\tau$ in the usual sense.  Note that members of $\mathcal{D}_{3\text{-}12}(n)$ or $\mathcal{D}_{3\text{-}21}(n)$  containing one cycle are synonymous with the members of $\mathcal{S}_{n-1}(3\text{-}12)$ or $\mathcal{S}_{n-1}(3\text{-}21)$.  Thus, $Z_1(x,1,v)$ and $R_1(x,1,v)$ are seen to give respectively the generating functions of the distributions on $\mathcal{S}_{n-1}(3\text{-}12)$ and $\mathcal{S}_{n-1}(3\text{-}21)$ for $n \geq 2$ for the parameter which tracks $n-1-j$, where $j\in[n-1]$ denotes the final letter of a permutation belonging to either avoidance class.  Finally, taking $u=v=1$ implies
$Z_1(x,1,1)=R_1(x,1,1)=x(Bell(x)-1)$, which confirms the well-known fact that $|\mathcal{S}_{n}(3\text{-}12)|=|\mathcal{S}_{n}(3\text{-}21)|=B_n$ for all $n \geq 1$.

%----------------------------------------------------------------------------------------------------

\end{document}